\documentclass[12pt]{amsart}
\usepackage[margin=1.35in]{geometry}
\usepackage{amscd,amsmath,amsxtra,amsthm,amssymb,stmaryrd,xr,mathrsfs,mathtools,enumerate,commath, comment}
\usepackage{stmaryrd}
\usepackage{multirow}
\usepackage{xcolor}
\usepackage{commath}
\usepackage{comment}
\usepackage{tikz-cd}
\usepackage{longtable} 
\usepackage{pdflscape} 
\usepackage{booktabs}
\usepackage{orcidlink}
\usepackage{hyperref}
\definecolor{vegasgold}{rgb}{0.77, 0.7, 0.35}
\definecolor{darkgoldenrod}{rgb}{0.72, 0.53, 0.04}
\definecolor{gold(metallic)}{rgb}{0.83, 0.69, 0.22}
\hypersetup{
 colorlinks=true,
 linkcolor=darkgoldenrod,
 filecolor=brown,      
 urlcolor=gold(metallic),
 citecolor=darkgoldenrod,
 }
\newtheorem{lthm}{Theorem}

\usepackage[all,cmtip]{xy}

\DeclareFontFamily{U}{wncy}{}
\DeclareFontShape{U}{wncy}{m}{n}{<->wncyr10}{}
\DeclareSymbolFont{mcy}{U}{wncy}{m}{n}
\DeclareMathSymbol{\Sh}{\mathord}{mcy}{"58}
\usepackage[T2A,T1]{fontenc}
\usepackage[OT2,T1]{fontenc}

\newtheorem{theorem}{Theorem}[section]
\newtheorem{lemma}[theorem]{Lemma}
\newtheorem{ass}[theorem]{Assumption}
\newtheorem*{theorem*}{Theorem}
\newtheorem*{ass*}{Assumption}
\newtheorem{definition}[theorem]{Definition}

\newtheorem{remark}[theorem]{Remark}

\newtheorem{conjecture}[theorem]{Conjecture}
\newtheorem{proposition}[theorem]{Proposition}

\newcommand{\cF}{\mathcal{F}}

\newcommand{\cE}{\mathcal{E}}

\newcommand{\cG}{\mathcal{G}}

\newcommand{\Z}{\mathbb{Z}}

\newcommand{\Q}{\mathbb{Q}}
\newcommand{\F}{\mathbb{F}}

\newcommand{\cC}{\mathcal{C}}

\newcommand{\op}[1]{\operatorname{#1}}

\newcommand\mtx[4] { \left( {\begin{array}{cc}
 #1 & #2 \\
 #3 & #4 \\
 \end{array} } \right)}

\numberwithin{equation}{section}

\begin{document}

\title[Diophantine Stability in Certain Galois Extensions]{A positive density of elliptic curves is diophantine stable in certain Galois Extensions}

\author[A.~Ray]{Anwesh Ray\, \orcidlink{0000-0001-6946-1559}}
\address[Ray]{Chennai Mathematical Institute, H1, SIPCOT IT Park, Kelambakkam, Siruseri, Tamil Nadu 603103, India}
\email{anwesh@cmi.ac.in}

\author[P.~Shingavekar]{Pratiksha Shingavekar\, \orcidlink{0000-0002-6903-2479}}
\address[Shingavekar]{Chennai Mathematical Institute, H1, SIPCOT IT Park, Kelambakkam, Siruseri, Tamil Nadu 603103, India}
\email{pshingavekar@gmail.com}


\begin{abstract}
Let $p \in \{3, 5\}$ and consider a cyclic $p$-extension $L/\mathbb{Q}$. We show that there exists an effective positive density of elliptic curves $E$ defined over $\mathbb{Q}$, ordered by height, that is \emph{diophantine stable} in $L$.
\end{abstract}

\maketitle

\section{Introduction}

\subsection{Background and motivation} 
\par Let $E$ be an elliptic curve over $\Q$ and $L$ be a number field. Following Mazur and Rubin \cite{Mazurrubindiophantine}, $E$ is said to be \emph{diophantine stable} if the natural inclusion induces an equality $E(L)=E(\Q)$. Mazur and Rubin show that if all $\bar{\Q}$-automorphisms of $E$ are defined over $\Q$, then there exist infinitely many prime-power cyclic extensions $L/\Q$ in which $E$ is diophantine stable. We refer to \cite[Theorem 1.2 and Theorem 1.3]{Mazurrubindiophantine} for further details.
\par Assume that $L/\Q$ is Galois with cyclic Galois group $G=\op{Gal}(L/\Q)$. Then the Hasse-Weil $L$-function of $E$ over $L$ decomposes as follows $\mathscr{L}_E(s,  L) = \prod_{\chi}\mathscr{L}_E(s, \chi)$,
where $\chi$ ranges over all Dirichlet characters $\chi: \widehat{G}\rightarrow \mathbb{C}^\times$. The Birch and Swinnerton--Dyer conjecture predicts that 
\[\op{rank} E(L)=\op{ord}_{s=1} \mathscr{L}_E(s, L)=\sum_{\chi\in \widehat{G}} \op{ord}_{s=1}\mathscr{L}_E(s,  \chi). \]There are heuristics for the growth and stability of ranks of elliptic curves $E$ in cyclic extensions, which are based on the relationship between modular symbols and the values of $\op{ord}_{s=1}\mathscr{L}_E(s,  \chi)$ above. For further details, we refer to \cite{Mazurrubinmodsymbols}. 

There has been an increasing interest in the study of diophantine stability and rank growth of a given elliptic curve in families of number field extensions, see for instance \cite{MazurRubingrowthofSelmer, Mazurrubingrowthnonabelian, Oliverthorne, ShnidmanWeiss, BKR, Berg}. Moreover, the study of such rank stability questions has close connections to Hilbert's tenth problem, see \cite{MazurRubinH10Inventiones}. These connections have recently been studied with rejuvenated interest by Garcia--Fritz and Pasten \cite{garcia-fritzpasten} and Kundu--Lei--Sprung \cite{KLS}.
\subsection{Main result}
\par In this article we formulate a new direction of enquiry. Instead of fixing an elliptic curve and varying the number field extension in a family, we shall fix an extension $L/\Q$ and count the number of elliptic curves $E_{/\Q}$, ordered by height, such that $E(L)=E(\Q)$. The result is proven only for cyclic extensions $L/\Q$ with $\op{Gal}(L/\Q)\simeq \Z/p\Z$, where $p\in \{3, 5\}$.

\begin{lthm}[Theorem \ref{main thm}]
    Let $p$ be a prime in $\{3, 5\}$ and $L/\Q$ be a Galois extension of $\Q$ with $\op{Gal}(L/\Q)\simeq \Z/p\Z$. Take $Z$ to denote the primes that ramify in $L$. Assume that 
    \begin{enumerate}
        \item $p\notin Z$ and $2$ splits completely in $L$, 
        \item for all $\ell \in Z$, there exists an elliptic curve $\mathbb{E}$ over $\F_\ell$, such that $\mathbb{E}(\F_\ell)[p]=0$ and $\mathbb{E}^{-1}(\F_\ell)[p]=0$. Here, $\mathbb{E}^{-1}$ is the quadratic twist of $\mathbb{E}$ by $-1$. 
    \end{enumerate}
    Then there is a positive density of elliptic curves $E_{/\Q}$ such that
    \[E(L)=E(\Q)=0. \]
    Evidently, these elliptic curves are \emph{diophantine stable} in $L$.
\end{lthm}

\begin{remark}
    At this point, several remarks are in order.
    \begin{itemize}
        \item The results rely on the techniques of Bhargava and Shankar \cite{BhargavaShankarbinary, bhargavashankar15, bhargavashankar5}, who study the average size of the $p$-Selmer group of elliptic curves $E_{/\Q}$. This is only done for primes $p=2, 3, 5$, however, the case when $p=2$ is omitted since our methods do not apply to this case. 
        \item The condition (2) can be dealt with via explicit calculation (cf. the code in Remark \ref{remark with code}). Indeed, when $p=5$ and $\ell$ is large enough, the condition (2) is shown to be be satisfied using results of Howe \cite{Howe}. We refer to Lemma \ref{howe lemma} for further details. 
        \item An explicit lower density is obtained for the curves $E_{/\Q}$ that are diophantine stable in $L$. This expression is quite fascinating and involves locally defined invariants associated to the extension $L/\Q$, see the statement of Theorem \ref{main thm} for details. For instance, when $p=3$ and $L\subset \Q(\mu_{31})$ is the cubic subfield, the lower density of elliptic curves $E_{/\Q}$ that are diophantine stable in $L$ is 
\[\geq \frac{1}{4} \times \frac{1}{2^{21}}\times \left(1-\frac{1}{3}\right)\times \frac{585}{961}\times \prod_{\ell\neq 2,3,31} \left(1-\frac{2}{\ell^2}+\frac{1}{\ell^3}\right).\]
    \end{itemize}
\end{remark}

\subsection{Organization}
\par Including the introduction, the article consists of four sections. In section \ref{s 2}, we study the stability of Selmer groups in prime cyclic extensions. In greater detail, let $p$ be an odd prime number and $L/K$ be an extension with $\op{Gal}(L/K)\simeq \Z/p\Z$. Suppose that the $p$-Selmer group of $E$ over $K$ is $0$. Then if certain additional local condition are satisfied for $E_{/K}$, then, it follows that the $p$-Selmer group of $E$ over $L$ is also $0$. The Galois cohomological arguments presented in this section are essentially due to \v{C}esnavi\v{c}ius \cite{navicius} and Brau \cite{brau2014selmer}. Section \ref{s 3} is preparatory in nature, and delineates the strategy that we take to prove our main result. We introduce both local and global density conditions on our families of elliptic curves that will play a role in our sieve theoretic arguments. Section \ref{s 4} is devoted to the proof of our main result. In this section, we describe a suitable large family of elliptic curves defined by local conditions and show that a positive proportion of elliptic curves in this family are indeed diophantine stable in $L$. The analysis leads to an explicit lower bound for the density of elliptic curves that are diophantine stable in our fixed extension $L$. This is precisely the step in which the aforementioned results of Bhargava and Shankar are applied.

\subsection*{Data availability} No data was analyzed in proving the results in the article.

\subsection*{Acknowledgment
} We thank the anonymous referee for the helpful report.

\section{Selmer groups and diophantine stability}\label{s 2}
\par In this section, we review the necessary background and notation required for our examination of elliptic curves over number fields. Some of the results in this section extend arguments due to \v{C}esnavi\v{c}ius \cite{navicius} and the unpublished work of Brau \cite{brau2014selmer}. Throughout, we denote by $K$ a number field and use $\Omega(K)$ (resp. $\Omega_\infty(K)$) to represent the set of finite primes (resp. archimedean places) of $K$. Given a prime number $\ell$, $\Omega_\ell(K)$ denotes the primes in $\Omega(K)$ that lie above $\ell$. Let $\overline{K}$ be an algebraic closure of $K$, and $\operatorname{G}_K$ denote the absolute Galois group $\operatorname{Gal}(\overline{K}/K)$. For each prime $v \in \Omega(K)$, we select an embedding $\iota_v: \overline{K} \hookrightarrow \overline{K}_v$. Setting $\operatorname{G}_{K_v} := \operatorname{Gal}(\overline{K}_v / K_v)$, the inclusion $\iota_v$ induces a Galois group inclusion $\iota_v^*: \operatorname{G}_K \hookrightarrow \operatorname{G}_{K_v}$.

\subsection{Selmer groups associated to elliptic curves}
\par Let $E$ be an elliptic curve defined over $K$. The Mordell-Weil group of $E$ over $K$, denoted $E(K)$, comprises the $K$-rational points on $E$. Choose an algebraic closure $\overline{K}$ of $K$, and let $E[p^n]$ be the $p^n$-torsion subgroup of $E(\overline{K})$. Let $\op{G}_K$ denote the absolute Galois group $\op{Gal}(\overline{K}/K)$ and consider the natural action of $\op{G}_K$ on $E[p^n]$. The Selmer group $\operatorname{Sel}_{p^n}(E/K)$ can be interpreted as a certain subgroup of the cohomology group $H^1(K, E[p^n])$, which measures the obstructions to representing $K$-rational points of $E$ using elements of $E[p^n]$. It is defined as the kernel to the natural restriction map in Galois cohomology
\[ \operatorname{Sel}_{p^n}(E/K) = \ker\left( H^1(K, E[p^n]) \to \bigoplus_v H^1(K_v, E)[p^n] \right), \]
where $v$ runs over all primes of $K$.

The $p^n$-Selmer group is closely related to the Mordell--Weil group and Tate--Shafarevich group.  Indeed, there is a natural short exact sequence that describes this relationship
\[0\rightarrow E(K)/p^n E(K)  \rightarrow \op{Sel}_{p^n}(E/K)\rightarrow \Sh(E/K)[p^n]\rightarrow 0.\]
Thus in particular, $\op{Sel}_{p^n}(E/K)$ is trivial if and only if 
\begin{itemize}
    \item $E(K)$ has rank $0$, 
    \item $E(K)[p]=0$, 
    \item $\Sh(E/K)[p^\infty]=0$.
\end{itemize}

\subsection{Growth and stability of Selmer groups}

\par Let $E_{/\Q}$ be an elliptic curve and $K$ be a number field. Let $p$ be an odd prime number. Fix a finite Galois extension $L/K$ for which $G:=\op{Gal}(L/K)$ is a $p$-group and assume that $G\simeq \Z/p\Z$. Set $k_v$ to denote the residue field of $K$ at $v$. If $v$ is a prime of good reduction, we set $\widetilde{E}(k_v)$ to be the group of $k_v$-rational points of the mod-$v$ reduction of $E$. Let $c_v(E/K)$ be the Tamagawa number of $E$ at $v$.

\begin{ass}\label{ass on good reduction at v in S}
Throughout, we assume that the following conditions hold.
\begin{enumerate}
    \item The Selmer group $\op{Sel}_p(E/K)=0$. In particular, this implies that the rank of $E(K)$ is $0$ and both $E(K)$ and $\Sh(E/K)$ have no nontrivial $p$-torsion.
    \item At all primes $v$ of $K$ that lie above $p$, assume that $E$ has good reduction. 
    \item If $v\in \Omega_p(K)$ is ramified in $L$, we additionally require that $E$ has ordinary reduction at $v$ and $\widetilde{E}(k_v)[p]=0$.
    \item For all primes $v\nmid p$ that are ramified in $L$, we require that $E$ has good reduction at $v$. Moreover, for such primes $v$, we have that $\widetilde{E}(k_v)[p]=0$.
    \item At all primes $v\nmid p$ at which $E$ has split multiplicative reduction, we require that $p\nmid c_v(E/K)$.
    \item The primes $v$ of $K$ of additive reduction split completely in $L$.
\end{enumerate}

\end{ass}
Define $S$ as the set of primes $v\in \Omega(K)$ satisfying at least one of the following conditions:
\begin{itemize}
    \item $v\in \Omega_p(K)$,
    \item $E$ has bad reduction at $v$,
    \item $v$ is ramified in $L$.
\end{itemize}
Let $S(L)$ denote the set of primes $w\in \Omega(L)$ that lie above some prime $v\in S$. The natural restriction map 
\[\alpha: \op{Sel}_p(E/K)\rightarrow \op{Sel}_p(E/L)^G\] fits within the following fundamental diagram
 \begin{equation}\label{fdiagram}
\begin{tikzcd}[column sep = small, row sep = large]
& 0=\op{Sel}_p(E/K) \arrow{r}\arrow{d}{\alpha} & H^1(K_S/K, E[p])\arrow{r}{\Phi_{K}} \arrow{d}{\beta} & \bigoplus_{v\in S} H^1(K_v, E)[p] \arrow{d}{\gamma}  \arrow{r} & 0\\
0\arrow{r} & \op{Sel}_p(E/L)^G\arrow{r} & H^1(L_S/L, E[p])^G\arrow{r}  &\left(\bigoplus_{w\in S(L)} H^1(L_w, E)[p]\right)^G.
\end{tikzcd}
\end{equation}
Here, $\beta$, and $\gamma$ are restriction maps on global and local cohomology respectively. Since it is assumed that $\op{Sel}_p(E/K)=0$, it follows from the Cassels-Poitou-Tate long exact sequence that $\Phi_K$ is surjective. For $v\in S$, the restriction map is denoted by
\[\gamma_v: H^1(K_v, E)[p]\rightarrow \left(\bigoplus_{w|v} H^1(L_w, E)[p]\right)^G,\] where $w$ runs through all primes of $L$ above $v$. We note that $\gamma$ is the direct sum $\bigoplus_{v\in S} \gamma_v$.

\begin{lemma}\label{lemma V^G=0 implies V=0}    Consider a finite dimensional $\F_p$-vector space $V$ and the finite $p$-group $G$ acting on $V$. If $V^G=0$, then it follows that $V=0$.\end{lemma}

\begin{proof}
The result follows via an easy application of Nakayama's lemma. For further details, cf. \cite[Lemma 4.2]{muller2024hilbert}. 
\end{proof}


\begin{proposition}\label{SelpL=0 propn}
    With respect to notation above, assume that $\gamma_v$ is injective for all $v\in S$. Then it follows that $\op{Sel}_p(E/L)=0$, and consequently, the rank of $E(L)$ is also 0.
\end{proposition}
\begin{proof}
  From \eqref{fdiagram}, we arrive at the following long exact sequence
\begin{equation}\label{snake lemma exact sequence}
    0=\op{ker}(\alpha)\rightarrow \op{ker}(\beta)\rightarrow \op{ker}(\gamma)\rightarrow \op{coker}(\alpha)\rightarrow \op{coker}(\beta)\rightarrow \op{coker}(\gamma).
\end{equation}
Since we assume that $\op{Sel}_p(E/K)=0$, it follows that $E(K)[p]=0$. From the inflation-restriction sequence
\[0\rightarrow H^1(L/K, E(L)[p])\rightarrow H^1(K_S/K, E[p])\xrightarrow{\beta} H^1(K_S/L, E[p])^G\rightarrow H^2(L/K, E(L)[p])\] 
it follows that $\beta$ is an isomorphism. This implies that $\op{ker}\gamma$ is isomorphic to the $\op{coker}\alpha$. Since the maps $\gamma_v$ are injective, so is $\gamma$. Hence, $\op{Sel}_p(E/L)^G=0$. The result then follows from Lemma \ref{lemma V^G=0 implies V=0}, which asserts that 	$\op{Sel}_p(E/L)=0$.
\end{proof}

\par We are now left with studying the maps $\gamma_v$ for $v\in S$, and the conditions under which they are injective. We first deal with the case in which $v\nmid p$. It follows from the Assumption \ref{ass on good reduction at v in S} that either
\begin{itemize}
    \item $v$ is totally ramified in $L$ and that $E$ has good reduction at $v$,
    \item $v$ is unramified in $L$ and $E$ has nonsplit multiplicative,
    \item $v$ splits completely in $L$ and $E$ has additive reduction at $v$,
    \item $v$ is unramified in $L$, $E$ has split multiplicative reduction at $v$ and $p\nmid c_v(E/K)$. 
\end{itemize}

\begin{lemma}\label{gamma_v injective for v nmid p good}
    Let $v\nmid p$ be a prime in $S$ at which $E$ has good reduction. Then $\gamma_v$ is injective.
\end{lemma}

\begin{proof}
    It follows from the definition of $S$ that $v$ must be totally ramified in $L$. Recall that Assumption \ref{ass on good reduction at v in S} requires that $\widetilde{E}(k_v)[p]=0$. The result then follows from \cite[Proposition 2.8]{pathak2024rank}.
\end{proof}

\begin{lemma}\label{gamma_v injective for v nmid p bad}
    Let $v\nmid p$ be a prime in $S$ at which $E$ has bad reduction. Then, $\gamma_v$ is injective. 
\end{lemma}
\begin{proof}
    \par First we consider the case in which $E$ has nonsplit multiplicative reduction at $v$. In this case, the result follows from \cite[Proposition 5.8]{brau2014selmer}. On the other hand, consider the case in which $E$ has split multiplicative reduction. Note that in light of Assumption \ref{ass on good reduction at v in S}, $v$ must be unramified in $L$ and $p\nmid c_v(E/K)$. Then, the result follows from \cite[Proposition 5.7, (ii)]{brau2014selmer}. Finally, consider the case in which $E$ has additive reduction at $v$. In this case, it is assumed that $v$ splits completely in $L$. In this case, $L_w=K_v$ for any prime $w$ which lies above $v$, and it is clear that $\gamma_v$ is injective.
\end{proof}

Next, we show that the maps $\gamma_v$ are injective for all primes $v|p$.  

\begin{lemma}\label{gamma_v injective for v mid p}
    Let $v|p$ be a prime (in $S$), then the map $\gamma_v$ is injective. 
\end{lemma}
\begin{proof}
    It follows from Assumption \ref{ass on good reduction at v in S} that $E$ has good reduction at $v$. There are three cases to consider. 
    \begin{itemize}
        \item First, assume that $v$ is unramified in $L$ and that $E$ has good ordinary reduction at $v$. In this case, $\gamma_v$ is injective by \cite[Proposition 5.9, (ii)]{brau2014selmer}.
        \item Next consider the case when $v$ is ramified in $L$ and $E$ has good ordinary reduction at $v$. In this case, it has been assumed that $\widetilde{E}(k_v)[p]=0$, by Assumption \ref{ass on good reduction at v in S}. In this case, $\gamma_v$ is injective, according to \cite[Proposition 5.9, (i)]{brau2014selmer}.
        \item Consider the case when $v$ is unramified in $L$ and $E$ has good supersingular reduction at $v$. Then, $\gamma_v$ is injective by \cite[Proposition 5.10, (ii)]{brau2014selmer}.
    \end{itemize}
     
\end{proof}

\begin{proposition}\label{p selmer over L = 0}
    Let $E$ be an elliptic curve satisfying the conditions of Assumption \ref{ass on good reduction at v in S}. Then, we have that $\op{Sel}_p(E/L)=0$. 
\end{proposition}
\begin{proof}
    By Proposition \ref{SelpL=0 propn}, it suffices to show that the maps $\gamma_v$ are injective for all $v\in S$. Lemma \ref{gamma_v injective for v nmid p good} (resp. \ref{gamma_v injective for v nmid p bad}) asserts that $\gamma_v$ is injective when $v\nmid p$ is a prime in $S$ of good (resp. bad) reduction for $E$. For primes $v|p$, $\gamma_v$ is injective by Lemma \ref{gamma_v injective for v mid p}. The result thus follows. 
\end{proof}

\section{Asymptotics for elliptic curves}\label{s 3}
\par In this section, we shall fix a prime $p\in\{3, 5\}$ and a $\Z/p\Z$-extension $L/\Q$. Given an elliptic curve $E_{/\Q}$, it is well known that up to isomorphism $E$ admits a short Weierstrass equation of the form \[E_{A,B}:y^2=x^3+Ax+B,\] where $(A, B)\in \Z^2$ are such that for all primes $\ell$, either $\ell^4 \nmid A$ or $\ell^6 \nmid B$. Such an equation is \emph{globally minimal} and the pair $(A,B)$ is uniquely determined. The discriminant of $E_{A,B}$ is given by 
\[\Delta_{A,B}=\Delta(E_{A,B})=-16(4A^3+27B^2),\] and the \emph{$j$-invariant} is \[j_{A,B}=j(E_{A,B})=\frac{2^{8}3^3 A^3}{4A^3+27B^2}.\]
We shall denote by $\cC$ the set of all pairs of integers $(A, B)$ such that $\ell^4\nmid A$ or $\ell^6\nmid B$ for all primes $\ell$. The association taking $(A,B)$ to the isomorphism class of $E_{A,B}$, gives a natural parametrization of isomorphism classes of elliptic curves $E_{/\Q}$.
\par This parametrization of elliptic curves allows us to formulate counting questions. Given an elliptic curve $E_{A,B}$ the \emph{naive height} of $E_{A,B}$ is defined as follows \[h_{A,B}=h\left(E_{A,B}\right):=\op{max} \{|A|^3, |B|^2\}.\] The elliptic curves we shall consider shall be ordered according to their height. Let us explain this further. Let $X>0$ be a real number and $\cC(X)$ to be the family of elliptic curves $E_{A,B}$ with height at most $X$, i.e.,
\[\cC(X):=\{(A,B)\in \cC: h\left(E_{A,B}\right) \le X\}.\] It is easy to see that the set $\cC(X)$ is finite, in fact Brumer proves a precise asymptotic. 

\begin{lemma}\label{brumer lemma}
Let $\cC(X)$ be given as above. One has the following asymptotic estimate for the growth of $\#\cC(X)$, as $X$ goes to $\infty$:
\[
\# \cC(X) = \frac{4X^{5/6}}{\zeta(10)} + O\left(\sqrt{X}\right).
\]
\end{lemma}
\begin{proof}
    For a proof of this result, see \cite[Lemma 4.3]{Bru92}.
\end{proof}

\begin{definition}
    Any set of isomorphism classes of elliptic curves over $\Q$ can be identified with a subset $S$ of $\cC$. Given $X>0$, set $S(X):=S\cap \cC(X)$. The density of $S$ (if it exists) is defined as follows
    \begin{equation}\label{limit def density}\mathfrak{d}(S):=\lim_{X\rightarrow \infty} \frac{\# S(X)}{\#\cC(X)}.\end{equation}
    When the limit in the above expression is replaced by $\limsup$ (resp $\liminf$), we write $\overline{\mathfrak{d}}(S)$ (resp. $\underline{\mathfrak{d}}(S)$) to denote the upper (resp. lower) density. The set $S$ is said to have \emph{positive density} if $\underline{\mathfrak{d}}(S)>0$.
\end{definition}

\begin{remark}
    It is worth noting that the density $\mathfrak{d}(S)$ is a quantity that need not be defined when the limit \eqref{limit def density} does not exist. However, the upper and lower densities are always well defined. 
\end{remark}

The rank of an elliptic curve is a fundamental and subtle invariant. It is natural to view the rank as a random function of the set of all elliptic curves and study its distribution. Katz and Sarnak modeled the distribution of ranks of elliptic curves and conjectured the following.

\begin{conjecture}[Rank distribution conjecture]
    The density of elliptic curves $E_{/\Q}$ with rank $0$ (resp. $1$) is $1/2$ (resp. $1/2$). 
\end{conjecture}
The conjecture thus also predicts that the density of elliptic curves with rank $\geq 2$ is $0$. Significant progress was made in seminal work of Bhargava and Shankar \cite{BhargavaShankarbinary, bhargavashankar15, bhargavashankar5}. The idea is to relate Selmer groups of elliptic curves to integral forms. The integral forms are lattice points in certain fundamental regions in Euclidean space and techniques from the geometry of numbers are applied to study their distribution. Given a prime $p$, it is expected that the average size of the $p$-Selmer group of $E_{/\Q}$ is $(p+1)$. This has been proven to be the case for primes $p\leq 5$. We refer to work of Poonen and Rains \cite{poonenrains} for predictions on the average size of Selmer groups and Tate--Shafarevich groups. More recently, Smith has proven various distribution results for the $2$-primary Selmer groups of twist families of abelian varieties. For further details, we refer to \cite{Smith1, Smith2, smith3}.

\par In this section, we fix a Galois extension $L/\Q$ with Galois group $\op{Gal}(L/\Q)\simeq \Z/p\Z$, where $p\in \{3,5\}$. Let $E_{/\Q}$ be an elliptic curve. Following Mazur and Rubin \cite{Mazurrubindiophantine}, $E$ is said to be \emph{diophantine stable} in $L$ if the natural inclusion yields an equality $E(\Q)=E(L)$. We assume throughout that $2$ and $p$ are unramified in $L$, and show that there is an effective positive density of elliptic curves $E_{/\Q}$ that are diophantine stable in $L$. This is done by showing that
\begin{itemize}
    \item the set of elliptic curves $E_{/\Q}$ for which $\op{E}(L)_{\op{tors}}=0$ has density $1$.
    \item The set of elliptic curves $E_{/\Q}$ for which $\op{rank}E(L)=0$ has positive density.
\end{itemize}

It then follows that the set of elliptic curves $E_{/\Q}$ for which both $E(\Q)=E(L)=0$ has positive density. We shall obtain an effective lower bound for the lower density of this set of elliptic curves.

\subsection{Controlling the torsion}
\par Consider the set $S_1\subset \cC$, of all isomorphism classes of elliptic curves $E_{/\Q}$ such that $E(L)_{\op{tors}}=0$. In this section, we show that $S_1$ has density $1$, i.e., 
\[\lim_{X\rightarrow\infty} \frac{\#S_1(X)}{\# \cC(X)}=1.\]An elliptic curve $E_{/\Q}$ gives rise to a family of Galois representations. Choose an algebraic closure $\overline{\Q}$ of $\Q$ and set $\op{G}_{\Q}:=\op{Gal}(\overline{\Q}/\Q)$. Given a natural number $n$, let $E[n]$ be the $n$-torsion subgroup of $E(\overline{\Q})$. Note that $E[n]\simeq \left(\Z/n\Z\right)^2$ and that there is a natural action of $\op{G}_{\Q}$ on $E[n]$. The associated Galois representation 
\[\rho_{E,n} : \op{G}_{\Q}\rightarrow \op{Aut}(E[n])\xrightarrow{\sim}\op{GL}_2(\Z/n\Z) \]
encodes many arithmetic properties of $E$. Suppose that $E$ does not have complex multiplication. Then, Serre's Open image theorem implies that for all but finitely many primes $\ell$, the mod-$\ell$ representation $\rho_{E, \ell}$ is surjective. On the other hand, a prime $\ell$ for which $\rho_{E, \ell}$ fails to be surjective is called an \emph{exceptional prime}. In this context, we set $S_1'\subset \cC$ to be the set of isomorphism classes of elliptic curves with no exceptional primes. 

\begin{theorem}[Duke \cite{dukeexceptional}]\label{Duke}
    The set $S_1'$ has density $1$, i.e., almost all isomorphism classes of elliptic curves (over $\Q$) have no exceptional primes. 
\end{theorem}

Thus, in order to show that $S_1$ has density $1$, it suffices to show that $S_1'$ is contained in $S_1$. In other words, if an elliptic curve $E_{/\Q}$ has the property that it has no exceptional primes, then, it follows that $E(L)_{\op{tors}}=0$. This assertion follows from the following result. 

\begin{lemma}\label{rho surj implies 0 torsion}
    Let $L/\Q$ be a $\Z/p\Z$-extension (where $p=3$ or $5$) and let $\ell$ be any prime number. Let $E_{/\Q}$ be an elliptic curve for which the representation \[\rho_{E,\ell}: \op{G}_{\Q}\rightarrow \op{GL}_2(\Z/\ell \Z)\] is surjective. Then, we have that $E(L)[\ell]=0$.
\end{lemma}

\begin{proof}
   \par By way of contradiction, assume that $E(L)[\ell] \neq 0$. Thus, $\rho_{E,\ell}(\op{G}_L)$ is contained in the mirabolic subgroup $M:= \left\{\mtx{1}{\ast}{0}{\ast}\right\}$ of $\op{GL}_2(\Z/\ell\Z)$. In particular, we find that
   \[[\op{GL}_2(\Z/\ell\Z):\rho_{E,\ell}(\op{G}_L)]\geq [\op{GL}_2(\Z/\ell\Z):M]=\frac{(\ell^2-\ell)(\ell^2-1)}{(\ell-1)\ell}=\ell^2-1.\]
   On the other hand, 
   \[[\op{GL}_2(\Z/\ell\Z):\rho_{E,\ell}(\op{G}_L)]= [\rho_{E,\ell}(\op{G}_\Q): \rho_{E,\ell}(\op{G}_L)]\leq [\op{G}_{\Q}: \op{G}_L]\leq 5.\]
   Thus, we find that $\ell^2-1\leq 5$, i.e., $\ell=2$. 
   
\par Assume therefore that $\ell=2$. Note that $\op{GL}_2(\Z/2\Z)\simeq S_3$, and $M\simeq C_2$. Since $\op{G}_L$ is a normal subgroup of $\op{G}_{\Q}$, it follows that $C_2$ is normal in $S_3$, which is a contradiction. Thus, we have shown that $E(L)[\ell]=0$. 
\end{proof}

\begin{theorem}\label{S_1 has density 1}
    The set $S_1$ has density $1$.
\end{theorem}
\begin{proof}
    Recall that Theorem \ref{Duke} asserts that the density of the set $S_1'$ is $1$. Hence, to prove the result, it suffices to show that $S_1' \subset S_1$. Let $E$ be an elliptic curve in $S_1'$. This means by definition, that for all primes $\ell$, the Galois representation $\rho_{E,\ell}:\op{G}_{\Q} \to \op{GL}_2(\Z/\ell\Z)$ is  surjective. Then by Lemma \ref{rho surj implies 0 torsion}, we get that $E(L)[\ell]=0$ for all primes $\ell$. This in turn implies that $E(L)_{\op{tors}}=0$ and hence $E \in S_1$. This completes the proof.
\end{proof}
\subsection{Controlling the rank}
\par In this section, we take $S_2\subset \cC$ to consist of isomorphism classes of elliptic curves $E_{/\Q}$ such that $\op{rank}E(L)=0$. Given a pair $(\overline{A}, \overline{B})\in (\Z/\ell\Z)^2$ with $4 \overline{A}^3+27 \overline{B}^2\neq 0$, let 
$E_{\overline{A}, \overline{B}}:y^2=x^3+\overline{A}x+\overline{B}$. Let $\ell$ be a prime that ramifies in $L$, then set $\mathfrak{A}_\ell$ to denote the number of pairs $(\overline{A}, \overline{B})\in (\Z/\ell\Z)^2$ such that
\begin{itemize}
\item $4 \overline{A}^3+27 \overline{B}^2\neq 0$, 
\item $E_{\overline{A}, \overline{B}}(\F_\ell)[p]=0$ and $E_{\overline{A}, -\overline{B}}(\F_\ell)[p]=0$.
\end{itemize}
We show in this section that $S_2$ has positive density, and note that the set $S_1\cap  S_2$ consists of isomorphism classes of elliptic curves $E_{/\Q}$ such that $E(L)=0$. Note that in particular, $E(\Q)=0$ and $E(L)=0$, hence $E$ is diophantine stable in the fixed extension $L/\Q$ for all $E\in S_1\cap S_2$. Theorem \ref{S_1 has density 1} asserts that $S_1$ has density $1$. Thus, if $S_2$ has positive density, then so does $S_1\cap S_2$. This then yields the following result. 
\begin{theorem}\label{main thm}
    Let $p$ be a prime in $\{3, 5\}$ and $L/\Q$ be a Galois extension of $\Q$ with $\op{Gal}(L/\Q)\simeq \Z/p\Z$. Take $Z$ to denote the primes that ramify in $L$. Assume that 
    \begin{enumerate}
        \item $p\notin Z$ and $2$ splits completely in $L$, 
        \item for all $\ell \in Z$, there exists an elliptic curve $\mathbb{E}$ over $\F_\ell$, such that $\mathbb{E}(\F_\ell)[p]=0$ and $\mathbb{E}^{-1}(\F_\ell)[p]=0$. Here, $\mathbb{E}^{-1}$ is the quadratic twist of $\mathbb{E}$ by $-1$. 
    \end{enumerate}
    Then there is a positive density of elliptic curves $E_{/\Q}$ such that
    \[E(L)=E(\Q)=0. \]
    Evidently, these elliptic curves are \emph{diophantine stable} in $L$. 
    Moreover, the lower density of this set of elliptic curves is at least $\eta_p \prod_\ell \delta_\ell$, where 
    \[\eta_p:=\begin{cases} \frac{1}{4} & \text{ if }p=3;\\
 \frac{3}{8} & \text{ if }p=5;\end{cases}\]
    and \[\delta_\ell:=\begin{cases}
      \frac{2}{3}  & \text{ if }\ell\notin Z\cup \{2,p\} \text{ and }\ell=3;\\
       1-\frac{2}{\ell^2}+\frac{1}{\ell^3}  & \text{ if }\ell\notin Z\cup \{2,p\} \text{ and }\ell\neq  3;\\
       1-\frac{1}{\ell} & \text{ if }\ell=p;\\
       \frac{1}{2^{21}} & \text{ if }\ell=2;\\
      \frac{\#\mathfrak{A}_\ell}{\ell^2} & \text{ if }\ell\in Z.
    \end{cases}\]
\end{theorem}

We note that (2) is a condition on the prime $\ell$, and it is shown that when $p=5$ and $\ell\geq 5779$ (resp. $p=3$ and $\ell\equiv 1\pmod{4}$), this condition is satisfied, see Lemma \ref{howe lemma}. We shall also give explicit lower bounds for the density of the set of elliptic curves that are diophantine stable in $L$ (in terms of $E$ and the set of primes that ramify in $L$). We shall show that $\op{Sel}_p(E/L)=0$ for a positive density of elliptic curves $E_{/\Q}$. This is achieved by applying Proposition \ref{p selmer over L = 0} by showing that Assumption \ref{ass on good reduction at v in S} (for $K:=\Q$ and $p\in \{3,5\}$) is satisfied for a positive density of elliptic curves. Recall that $p$ is unramified in $L$, and hence, Assumption \ref{ass on good reduction at v in S} (3) is satisfied.

\begin{definition}
    Given any prime $\ell$, let $\Sigma_\ell$ be a choice of a closed subset of $\Z_\ell^2$ defined by a congruence condition. In other words, there is a sufficiently large integer $n_\ell\geq 1$ and $\overline{\Sigma}_\ell\subseteq \left(\Z/\ell^{n_\ell} \Z\right)^2$ such that $(A, B)\in \Sigma_\ell$ if and only if \[\left(A\pmod{\ell^{n_\ell}}, B\pmod{\ell^{n_\ell}}\right)\in \left(\Z/\ell^{n_\ell}\Z\right)^2.\] Let $E_{A,B}:y^2=x^3+Ax+B$ be an elliptic curve such that either $\ell^4 \nmid A$ or $\ell^6 \nmid B$. Assume that choice of $\Sigma_\ell$ is made for every prime number. Then, we set $\Sigma$ to denote the tuple of conditions $(\Sigma_\ell)_{\ell}$, where $\ell$ ranges over all prime numbers. We associate the family $\mathcal{F}_\Sigma$ of elliptic curves to the tuple $(\Sigma_\ell)_\ell$, where $E_{A,B} \in \mathcal{F}_\Sigma$ if $(A,B) \in \Sigma_\ell$ for all $\ell$. The family $\mathcal{F}_\Sigma$ is then said to be defined by congruence conditions.
\end{definition}
Consider the following example. Given a prime $\ell$, we let $\overline{\Pi}_\ell\subset (\Z/\ell^6\Z)^2$ consist of all tuples $(\overline{A}, \overline{B})$ such that $\ell^4\nmid \overline{A}$ and $\overline{B}\neq 0$. Let $\Pi_\ell\subset \Z_\ell^2$ be the set of all tuples $(A,B)$ such that 
$\left(A\pmod{\ell^6}, B\pmod{\ell^6}\right)\in \overline{\Pi}_\ell$.
Thus, the set \[\Pi_\ell=\{(A, B)\in \Z_\ell^2 \mid \text{ either } \ell^4\nmid A \text{ or } \ell^6\nmid B\}\] corresponds to the minimal Weierstrass models of elliptic curves over $\Z_\ell$.

\par If $\mathcal{F}$ is a family of elliptic curves defined by congruence conditions, then we denote by $\op{Inv}(\mathcal{F})$ the subset of $\Z_\ell^2$ defined by \[\op{Inv}(\mathcal{F}):=\{(A,B): E_{A,B} \in \mathcal{F}\}\] and let $\big(\op{Inv}(\mathcal{F})\big)_\ell$ be its $\ell$-adic closure in $\Z_\ell^2$.
\begin{definition}
    A family $\mathcal{F}$ of elliptic curves defined by congruence conditions is said to be large if the set $\big(\op{Inv}(\mathcal{F})\big)_\ell$ contains all pairs $(A,B) \in \Z_\ell^2$ such that $\ell^2 \nmid \Delta(E_{A,B})$, for all but finitely many primes $\ell$.
\end{definition}
Recall that the \emph{root number} $\omega(E)$ of an elliptic curve $E_{/\Q}$ is the sign of the functional equation of the $L$-function $L(E,s)$ of $E$. It is widely believed that when the elliptic curves are ordered by height, the root numbers $+1$ and $-1$ occur equally often. The following result of Bhargava and Shankar \cite[Theorem 41]{bhargavashankar15} then gives the density of elliptic curves with rank $0$.
\begin{theorem}\label{thm of bhargava shankar}
   Let $p$ be a prime in $\{3,5\}$. Suppose $\mathcal{F}$ is a large family of elliptic curves such that exactly $50\%$ of the curves in $\mathcal{F}$, when ordered by height, have root number $+1$. Then if $p=3$ (resp. $p=5$) at least $25\%$ (resp. $37.5\%$) of the curves $E_{A,B}$ in $\mathcal{F}$, when ordered by height, have $\op{Sel}_p(E_{A,B}/\Q)=0$.
    %
\end{theorem}

\begin{proof}
    The statement of the result above is a little stronger than that of \cite[Theorem 41]{bhargavashankar15} (resp. \cite[Proposition 40]{bhargavashankar5}) since it is asserted that $\op{Sel}_p(E_{A,B}/\Q)=0$. The proof of the mentioned Theorems rely on showing that this Selmer group is indeed $0$ for at least $25\%$ (resp. $37.5\%$) of curves in $\cF$. When $p=3$, we refer to p. 617, ll. 11-12 of \cite{bhargavashankar15}. On the other hand, for $p=5$, see \cite[p.29, l. -7]{bhargavashankar5}. 
\end{proof}

For an elliptic curve $E:=E_{A,B}:y^2=x^3+Ax+B$ with $A, B \in\Z$, denote by $E^{-1}:=E^{-1}_{A,B}:y^2=x^3+Ax-B$ the twist by $-1$ of $E_{A,B}$. Note that $\Delta(E^{-1})=\Delta(E)$ and $j(E^{-1})=j(E)$. 

\par The global root number is a product of local root numbers 
\[\omega(E)=-\prod_\ell \omega_\ell(E).\] The local root number at a prime $\ell$ of multiplicative reduction is easy to describe. One has that 
\begin{equation}\label{local root number formula}\omega_\ell(E)=\begin{cases}
    +1 & \text{ if }E\text{ has good reduction or non-split multiplicative reduction at }\ell;\\
    -1 & \text{ if }E\text{ has split multiplicative reduction at }\ell.
\end{cases}\end{equation}
There is a simple criterion for this at a prime $\ell \ge 3$ of multiplicative reduction. The elliptic curve $E=E_{A,B}$ has split multiplicative reduction precisely when $\left(\frac{6B}{\ell}\right)=1$ (see \cite[Lemma 10]{Wong}). One has that $\left(\frac{-1}{\ell}\right)=1$ if and only if $\ell \equiv 1 \pmod 4$. Therefore, we deduce that for any prime $\ell \geq 3$ at which $E$ has multiplicative reduction, \begin{equation}\label{omega for mult redn}\omega_\ell(E^{-1})=\omega_\ell(E)\Leftrightarrow \ell \equiv 1 \pmod 4.\end{equation}
\begin{definition}\label{our defn of large family}
    Let $Z$ be the set of primes $\ell$ of $\Q$ that ramify in $L$. Note that $2,p\notin Z$ by assumption. Also, let $\mathcal{E}$ be the family of elliptic curves defined by the following conditions 
\begin{enumerate}
    \item $E$ has good reduction at $p$;
    \item $E$ has good reduction at every prime $\ell\in Z$. Moreover, $\widetilde{E}(\F_\ell)[p]=0$ and  $\widetilde{E}^{-1}(\F_\ell)[p]=0$ for all $\ell \in Z$.
    \item Let $\Delta'(E)$ be the positive prime to $2$ part of $\Delta(E)$ (i.e., $\Delta'(E):=\frac{|\Delta(E)|}{2^{v_2(\Delta(E))}}$). Then, $\Delta'(E)$ is squarefree and $\Delta'(E) \equiv 1\pmod{4}$.
    \item Both $E$ and $E^{-1}$ have additive reduction at $2$ and $2\nmid j(E)$.
\end{enumerate}
Let $\cE^0$ be the subset of $\cE$ consisting of isomorphism classes of elliptic curves $E\in \cE$ such that $\op{Sel}_p(E/\Q)=0$. 
\end{definition}

\begin{remark}
    It follows from the conditions above that if $E$ is an elliptic curve in  $\cE$, then $E^{-1}\in \cE$ as well. This is simply because $\Delta$, $\Delta'$ and $j$ are the same for $E$ and $E^{-1}$. 
\end{remark}

\begin{proposition}\label{ass 2.1 is satisfied prop}
    Let $E$ be an elliptic curve in $\mathcal{E}^0$. Then, setting $K:=\Q$, the Assumption \ref{ass on good reduction at v in S} is satisfied for $\mathcal{E}^0$.
\end{proposition}
\begin{proof}
All the conditions of Assumption \ref{ass on good reduction at v in S} simply follow from the definition of $\cE^0$.
\end{proof}

\begin{lemma}\label{lemma for root numbers}
    For $E\in \cE$, we have that $\omega(E^{-1})=-\omega(E)$.
\end{lemma}
\begin{proof}
Recall that $\Delta(E^{-1})=\Delta(E)$ and $j(E^{-1})=j(E)$. Recall also that the root number $\omega(E)=-\prod_\ell \omega_\ell(E)$. We study the relationship between $\omega_\ell(E^{-1})$ and $\omega_\ell(E)$ for all primes $\ell$. 
\begin{itemize}
    \item  We begin with primes $\ell$ not dividing $\Delta(E)$. Since $\Delta(E^{-1})=\Delta(E)$, both $E$ and $E^{-1}$ have good reduction at $\ell$. Thus, $\omega_\ell(E^{-1})=\omega_\ell(E)=1$.
    \item Next, let $\ell\neq 2$ be a prime dividing $\Delta(E)$. Recall that condition (3) in Definition \ref{our defn of large family} requires that $\Delta'(E)$ is squarefree. In particular, $\ell^2 \nmid \Delta(E)$, and thus, $E$ has multiplicative reduction at $\ell$. According to \eqref{omega for mult redn} we have that \[\omega_\ell(E^{-1})=\omega_\ell(E)\text{ if and only if }\ell \equiv 1 \pmod 4.\]
    Since it is assumed that $\Delta'(E)\equiv 1 \pmod 4$ (by part (3) of Definition \ref{our defn of large family}), we must have that the number of primes $\ell \equiv 3 \pmod 4$ dividing $\Delta(E)$ is even. Hence, $\omega_\ell(E^{-1})=-\omega_\ell(E)$ for an even number of primes $\ell \equiv 3 \pmod 4$. Thus, we find that 
    \[\prod_{\substack{\ell\neq 2,\\ \ell| \Delta(E)}}\omega_\ell(E^{-1})=\prod_{\substack{\ell\neq 2,\\ \ell| \Delta(E)}}\omega_\ell(E).\]
    \item Finally, let us consider $\ell=2$. By part (4) of Definition \ref{our defn of large family}, $E$ has additive reduction at $2$ and $2\nmid j(E)$. It follows from \cite[Lemma 12]{Wong} that $\omega_2(E^{-1})=-\omega_2(E)$.
\end{itemize} It follows from the analysis above that for $E\in \cE$,
\[\begin{split}\omega(E^{-1})= & -\prod_{\ell}\omega_\ell(E^{-1}) \\
=& -\prod_{\ell\nmid \Delta(E)}\omega_\ell(E^{-1})\times \prod_{\substack{\ell\neq 2,\\ \ell| \Delta(E)}}\omega_\ell(E^{-1})\times \omega_2(E^{-1})\\
=& -\prod_{\ell\nmid \Delta(E)}\omega_\ell(E)\times \prod_{\substack{\ell\neq 2,\\ \ell| \Delta(E)}}\omega_\ell(E^{-1})\times \left(-\omega_2(E^{-1})\right)\\
=& -\omega(E).
\end{split}
\]
\end{proof}

Recall that the \emph{height of} $E$, denoted $h(E)$, is defined to be $\op{max}\{|A|^3,B^2\}$. Thus, $h(E)$ remains unaltered under the twist by $-1$, i.e. $h(E)=h(E^{-1})$. Hence, Lemma \ref{lemma for root numbers} implies that exactly $50\%$ of the curves $E$ in $\cE$ have root number $+1$, when ordered by height. We note that there is a natural short exact sequence 
\[0\rightarrow E(\Q)/3 E(\Q)\rightarrow \op{Sel}_3(E/\Q)\rightarrow \Sh(E/\Q)[3]\rightarrow 0,\] and thus $\cE^0$ precisely consists of elliptic curves $E\in \cE$ for which 
\begin{enumerate}
    \item $\op{rank}E(\Q)=0$, 
    \item $E(\Q)[3]=0$, 
    \item $\Sh(E/\Q)[3]=0$. 
\end{enumerate}

\section{Density results for diophantine stability in $L$}\label{s 4}
\par This section is devoted to the key technical arguments of the article, finally leading up to the proofs of our main results.
\subsection{A suitable large family of elliptic curves}
\par We wish to show that the family $\cE$ (cf. Definition \ref{our defn of large family}) has positive lower density. In order to do this, we introduce a large family $\cF$ contained in $\cE$ and estimate the density of $\cF$. 

\begin{definition}
    Let $\cF\subset \cC$ consist of all elliptic curves $E=E_{A,B}$, where $(A,B)$ is a minimal pair such that 
    \begin{enumerate}
        \item $p\nmid 4A^3+27B^2$, 
        \item for all $\ell\in Z$, we require that $\ell\nmid 4A^3+27B^2$ and $\widetilde{E}_{A,B}(\F_\ell)[p]=0$ and $\widetilde{E}_{A,-B}(\F_\ell)[p]=0$, 
        \item $A=4A'$ and $B=16B'$, where $A' \equiv 189 \pmod{256}$ and $B' \equiv \pm 1 \pmod{256}$.
        \item For all primes $\ell\neq 2$, we insist that $\ell^2\nmid \Delta_{A,B}$.
    \end{enumerate}
\end{definition}
Clearly, the set $\cF$ is a large family, and at the primes $T:=Z\cup \{2, p\}$, we have associated congruence conditions. 

\begin{lemma}\label{F is subset of E}
    With respect to notation above, $\cF$ is a subset of $\cE$. 
\end{lemma}
\begin{proof}
    It is clear from the definition of $\cF$ that the conditions (1) and (2) in Definition \ref{our defn of large family} are satisfied. Also, $\Delta'(E)$ is clearly squarefree for $E\in \cF$. It thus suffices to show that
    \begin{itemize}
        \item $E$ and $E^{-1}$ have additive reduction at $2$, 
        \item $\Delta'(E)\equiv 1\pmod{4}$, 
        \item $2\nmid j(E)$
    \end{itemize}
     for all $E\in \cF$. That $E$ and $E^{-1}$ have additive reduction at $2$ follows from Tate's algorithm (see \cite[step 7, p. 367]{SilvermanAdvancedtopics}). Since it is a standard procedure, we omit these details. The $j$-invariant is given as follows
\[
j(E) = 1728 \frac{4 A^3}{4A^3 + 27B^2} = \frac{27 \times 2^{6} \times A'^3}{A'^3 + 27 B'^2}.
\]
We find that
\[
A'^3 + 27 B'^2 \equiv 189^3 + 27 \equiv 64 \pmod{256},
\]
and deduce that $j(E)$ is a $2$-adic unit. Lastly, note that
\[\Delta(E) = 4 A^3 + 27 B^2 = 2^8 (A'^3 + 27 B'^2),\]
and
\[A'^3 + 27 B'^2 = 2^6 (4n + 1),\]
where $n$ is a positive natural number. Thus, $\Delta'(E)=4n+1$ and is congruent to $1\pmod{4}$.
\end{proof}

\par Let us now spell out the congruence conditions that define $\cF$ as a subset of $\cC$.
\begin{itemize}
    \item For $\ell\notin T$, we let $\cF_\ell$ to consist of all Weierstrass models $E_{A,B}$ (not necessarily minimal) such that $\Delta_{A,B}\not\equiv 0\pmod{\ell^2}$. 
    \item For $\ell\in Z$, $\cF_\ell$ consist of $E_{A,B}$ such that $\Delta_{A,B}\not \equiv 0\pmod{\ell}$ and $\widetilde{E}_{A,\pm B}(\F_\ell)[p]=0$. 
    \item For $\ell=p$, $\cF_p$ is defined by $\Delta_{A,B}\not\equiv 0\pmod{p}$. 
    \item For $\ell=2$, $\cF_2$ is defined by requiring that $A=4A'$ and $B=16B'$ such that $(A', B')\equiv (189, \pm 1)\pmod{256}$.
\end{itemize}
 Thus, we find that $\cF=\cap_\ell \cF_\ell$, where $\ell$ ranges over all prime numbers. It is clear from our definition that for $\ell\neq 2$, the elliptic curves $E\in \cF_\ell$ are minimal at $\ell$. On the other hand, for $\ell=2$, minimality at $2$ follows from Tate's algorithm. Thus, we find that $\cF$ consists of elliptic curves that are globally minimal (i.e., minimal at all primes $\ell$). In other words, $\cF$ is a subset of $\cC$. We wish to compute the density of $\cF$. This is done in 2 steps. 
\begin{enumerate}
    \item Given $X>0$, we set $\cF_\ell(X):=\left\{(A, B)\in \cF_\ell\mid \op{max}\{|A|^3, B^2\}\leq X\right\}$. We note that not all pairs $(A, B)\in \cF_\ell$ are minimal, but the pairs $(A, B)\in \cF$ are. We derive an asymptotic formula for $\# \cF_\ell(X)$ as $X\rightarrow \infty$. 
    \item Next, we show that the asymptotic formulae for $\#\cF_\ell(X)$ can be used to give an asymptotic formula for $\#\cF(X)$. Using this formula, we derive an expression for the density of $\cF$. Since $\cF$ is contained in $\cE$, this gives us lower bound for the lower density of $\cE$.
\end{enumerate}

In order to complete step (1) in an effective manner, we introduce the residue classes $\overline{\cF}_\ell\subset (\Z/\ell^{n_\ell})^2$ associated to the local condition $\cF_\ell$ for all primes. We then set $\delta_\ell:= \frac{\#\overline{\cF}_\ell}{\ell^{2n_\ell}}$. The asymptotic for $\#\cF_\ell(X)$ is then given by 
\[\# \cF_\ell(X)\sim \frac{\# \overline{\cF}_\ell}{\ell^{2n_\ell}}\times \# \{(A, B)\in \Z^2\mid \op{max}\{|A|^3, B^2\}\leq X\}\sim 4\delta_\ell X^{5/6}. \]
\begin{definition}
    The set $\overline{\cF}_\ell$ is defined as follows.
    \begin{itemize}
        \item For $\ell\notin T$, we set $n_\ell=2$ and $\overline{\cF}_\ell\subset (\Z/\ell^2\Z)^2$ to consist of pairs $(\overline{A},\overline{B})$ for which $4\overline{A}^3+27 \overline{B}^2\neq 0$. 
        \item For $\ell\in Z$, take $n_\ell=1$ and $\overline{\cF}_\ell$ consist of $(\overline{A}, \overline{B})$ such that $4\overline{A}^3+27 \overline{B}^2\neq 0$ and $E_{\overline{A}, \pm\overline{B}}(\F_\ell)[p]=0$.
        \item For $\ell=p$, take $n_\ell=1$ and $\overline{\cF}_\ell$ consist of $(\overline{A}, \overline{B})$ such that $4\overline{A}^3+27 \overline{B}^2\neq 0$.
        \item For $\ell=2$, we set $n_\ell=12$ and $\overline{\cF}_\ell\subset (\Z/2^{12}\Z)^2$ to consist of pairs $(\overline{A},\overline{B})$ where $(\overline{A},\overline{B})=(756, \pm 16)$. 
    \end{itemize}
\end{definition}
For primes $\ell\notin Z$, we can calculate $\delta_\ell$ explicitly. 

\begin{lemma}\label{cl lemma}
    Let $\ell$ be a prime, we find that $\# \cF_\ell(X)\sim 4 \delta_\ell X^{5/6}$, where
    \[\delta_\ell= \begin{cases}
      \frac{2}{3}  & \text{ if }\ell\notin T \text{ and }\ell=3;\\
       1-\frac{2}{\ell^2}+\frac{1}{\ell^3}  & \text{ if }\ell\notin T \text{ and }\ell\neq  3;\\
       1-\frac{1}{\ell} & \text{ if }\ell=p;\\
       \frac{1}{2^{21}} & \text{ if }\ell=2.\\
    \end{cases}\]
\end{lemma}

\begin{proof}
    Each of these conditions is described by congruence conditions on $(A, B)\in \Z^2$. In other words, there is an integer $n_\ell>0$ and a set of congruence classes
    \[\overline{\cF}_\ell\subset (\Z/\ell^{n_\ell})^2\] such that $\cF_\ell$ consists of all pairs $(A,B)$ such that \[(A\pmod{\ell^{n_\ell}}, B\pmod{\ell^{n_\ell}})\in \overline{\cF}_\ell.\] It is thus clear that 
    \[\# \cF_\ell(X)\sim \left(\frac{\# \overline{\cF}_\ell}{\ell^{2n_\ell}}\right)\times 4 X^{5/6}.\] 
    We count $\# \overline{\cF}_\ell$ in all cases.
    \begin{itemize}
        \item First consider $\ell\notin T$. Recall that for $\ell\notin T$, $\cF_\ell$ consists of all pairs $(A,B)\in \Z^2$ such that $\ell^2\nmid \Delta_{A,B}$. Thus, $\overline{\cF}_\ell$ consists of all pairs $(\overline{A}, \overline{B})\in (\Z/\ell^2)^2$ such that $-4 \overline{A}^3\neq 27 \overline{B}^2$. Note that $\ell\neq 2$ since $2\in T$. First assume that $\ell=3$ (this case potentially arises when $p=5$). In this case, the condition becomes $3\nmid \overline{A}$, and thus, $\# \overline{\cF}_\ell=2\times 3^3$ and thus, 
        \[\left(\frac{\# \overline{\cF}_\ell}{\ell^{2n_\ell}}\right)=\left(\frac{\# \overline{\cF}_\ell}{\ell^{4}}\right)=\frac{2}{3}.\]
        Next, we move on the case when $\ell\neq 3$. Set $(C,D):=(-\overline{A}/3,\overline{B}/2)$ and count the number of pairs for which $C^3= D^2$. We consider two further cases. In the first case, $\ell\nmid C$, we set $(U,V):=(C, D/U)$ and the equation becomes $U=V^2$. Thus, in this case, there are $(\ell^2-\ell)$ values of $V$ and $U$ is determined by $V$. On the other hand, if $\ell|C$ then $\ell| D$ as well and there are $\ell^2$ choices in this case. Thus, we get 
        \[\#\overline{\cF}_\ell=\ell^4-(\ell^2-\ell)-\ell^2=\ell^4-2\ell^2+\ell.\]
        From the above, we find that 
        \[\left(\frac{\# \overline{\cF}_\ell}{\ell^{2n_\ell}}\right)=\left(\frac{\# \overline{\cF}_\ell}{\ell^{4}}\right)=1-\frac{2}{\ell^2}+\frac{1}{\ell^3}.\]
        \item Next, consider the case when $\ell=p$. In this case, the condition is that $\ell\nmid \Delta_{A,B}$. Arguing as we did in the previous case, it is easy to see that 
        \[\left(\frac{\# \overline{\cF}_\ell}{\ell^{2n_\ell}}\right)=\left(\frac{\# \overline{\cF}_\ell}{\ell^{2}}\right)=\frac{(\ell^2-\ell)}{\ell^2}=1-\frac{1}{\ell}.\]
        \item Recall that when $\ell=2$, the condition required that $A=4A'$ and $B=16B'$, where $A' \equiv 189 \pmod{256}$ and $B' \equiv \pm 1 \pmod{256}$. Thus, we find that
\[\frac{\# \cF_\ell}{\ell^{2n_\ell}}=\frac{1}{2^{21}}.\]
    \end{itemize}
\end{proof}

\begin{lemma}\label{C_1 tech lemma}
    Let $z\geq 5$ be a real number $\cF_\ell'$ be the complement of $\cF_\ell$ in $\Z^2$. There is an absolute constant $C_1>0$ (independent of $\ell$) such that for all large enough values of $X$,
    \[\# \left(\bigcup_{\ell> z}\cF_\ell'(X)\right)\leq \sum_{\ell>z}\frac{C_1}{\ell^2} X^{5/6}+o(X^{5/6}).\] 
\end{lemma}
\begin{proof}
Assume without loss of generality that $\ell\geq 5$. We refer to the first part of the proof of Lemma \ref{cl lemma}. Thus, $\cG_z=\bigcup_{\ell>z}\cF_\ell'$ is the set of $(A,B)\in \Z^2$ such that $\ell^2|(4A^3+27B^2)$ for some prime $\ell>z$. We may write $\cG_z=\mathcal{B}_1\cup \mathcal{B}_2\cup \mathcal{B}_3$, where 
\begin{itemize}
    \item $\mathcal{B}_1$ consists of pairs $(A,B)$ for which $\ell|A$ and $\ell|B$ for some prime $\ell>z$, 
    \item $\mathcal{B}_2$ consists of pairs for which $\ell\nmid A, B$ and $-4A^3\equiv 27B^2\pmod{\ell^2}$, and $4A^3+27B^2\neq 0$ for some prime $\ell>z$.
    \item Finally, let $\mathcal{B}_3$ consist of the pairs for which $4A^3+27B^2=0$. 
\end{itemize}

It is clear that \[\begin{split}\# \mathcal{B}_1(X) = & \{(A,B)\in \mathcal{B}_1\mid \op{max}\{|A|^3, B^2\}\leq X\} \\ 
\leq & \sum_{\ell>z} 4 \left\lfloor \frac{X^{1/3}}{\ell}\right \rfloor \left\lfloor \frac{X^{1/2}}{\ell}\right\rfloor \leq \sum_{\ell>z} \frac{4X^{5/6}}{\ell^2}.\end{split}\] Now suppose that $(A,B)\in \mathcal{B}_2(X)$. Note that $|4A^3+27B^2|\leq 31 X$ is nonzero and is divisible by $\ell^2$ for some $\ell>z$. Therefore, we find that $\ell\leq \sqrt{31 X}$. For each $A$ there are at most $2$ residue classes of $B$ modulo $\ell$ which satisfy $-4A^3\equiv 27B^2\pmod{\ell}$. Each residue class $B\pmod{\ell}$ lifts to a unique residue class modulo $\ell^2$ that satisfies $-4A^3\equiv 27B^2\pmod{\ell^2}$, by Hensel's lemma. This requires that $B\not\equiv 0\pmod{\ell}$. The number of such $B$s in the interval $[-X^{1/2}, X^{1/2}]$ in a given residue class modulo $\ell^2$ is $\frac{2 X^{1/2}}{\ell^2}+O(1)$. Thus, in all, \[\begin{split}\#\mathcal{B}_2(X)\leq & \sum_{z<\ell\leq \sqrt{31 X}} 2X^{1/3} \left(\frac{2 X^{1/2}}{\ell^2}+O(1)\right),\\
\leq & \sum_{\ell>z} \frac{4 X^{5/6}}{\ell^2}+O\left(X^{1/3}\pi(\sqrt{31 X})\right),\\
= & \sum_{\ell>z} \frac{4 X^{5/6}}{\ell^2} +O\left(\frac{X^{5/6}}{\log X}\right).
\end{split}\] Hence, $\# \mathcal{B}_2(X)\leq \sum_{\ell>z} \frac{4 X^{5/6}}{\ell^2} +o\left(X^{5/6}\right)$. Finally, it is easy to see that $\# \mathcal{B}_3(X)=o(X^{5/6})$ since the value of $A$ determines $B$ for $(A,B)\in \mathcal{B}_3(X)$. Thus, taking any constant $C_1\geq 8$ will work. 
\end{proof}

Next, we treat primes $\ell\in Z$, i.e., primes that ramify in $L$. By assumption, $\ell\notin \{2, p\}$. Recall that $\overline{\cF}_\ell$ is the subset of $(\Z/\ell\Z)^2$ consisting of pairs $(\overline{A}, \overline{B})$ such that 
\begin{itemize}
    \item $4 \overline{A}^3+27 \overline{B}^2\neq 0$, 
    \item $E_{\overline{A}, \overline{B}}(\F_\ell)[p]=0$ and $E_{\overline{A}, -\overline{B}}(\F_\ell)[p]=0$.
\end{itemize}
Note that \[\# \cF_\ell(X)\sim 4\delta_\ell X^{5/6},\] where $\delta_\ell=\frac{\# \overline{\cF}_\ell}{\ell^2}$. We estimate the size of $\# \overline{\cF}_\ell$ for $\ell\in Z$.

\begin{lemma}\label{howe lemma}
    With respect to notation above, the following assertions hold. 
    \begin{enumerate}
        \item For $\ell\in Z$, suppose that $p=5$. Then for $\ell\geq 5779$, we have that 
    \[\# \overline{\cF}_\ell\geq \left(\frac{\ell^2-\ell}{2}\right)\left( \frac{\sqrt{\ell}-76}{\sqrt{\ell}}\right). \]
    \item Suppose that $p=3$ and that the only primes $\ell$ that ramify in $L$ are of the form $\ell\equiv 1\pmod{4}$. Then for all primes $\ell\geq 233$, we have that 
    \[\# \overline{\cF}_\ell\geq \left(\frac{\ell^2-\ell}{2}\right)\left(\frac{\sqrt{\ell}-15.18}{\sqrt{\ell}}\right).\]
    \end{enumerate}

\end{lemma}
\begin{proof}
    The complement of $\overline{\cF}_\ell$ in $(\Z/\ell\Z)^2$ can be expressed as a union $\mathcal{A}_1\cup \mathcal{A}_2\cup \mathcal{A}_3$. Here, $\mathcal{A}_1$ consists of pairs $(\overline{A}, \overline{B})\in (\Z/\ell\Z)^2$ such that $\Delta_{\overline{A}, \overline{B}}=0$ and $\mathcal{A}_2$ (resp. $\mathcal{A}_3$) consists of pairs for which $E_{\overline{A}, \overline{B}}(\F_\ell)[p]\neq 0$ (resp. $E_{\overline{A}, -\overline{B}}(\F_\ell)[p]\neq 0$). Thus, we find that 
     \[\# \cF_\ell\geq \ell^2-\# \mathcal{A}_1-\# \mathcal{A}_2-\# \mathcal{A}_3 .\]It is clear that 
     \[\# \mathcal{A}_2=\# \mathcal{A}_3.\]As was discussed in the proof of Lemma \ref{cl lemma}, we have that $\# \mathcal{A}_1=\ell$. Then it follows from \cite{Howe} that 
     \[\# \mathcal{A}_2\leq  \left(\frac{\ell^2-\ell}{p-1}\right)\left(1+2.53 \frac{p(p+1)}{\sqrt{\ell}}\right).\]
     Therefore, we have that 
     \[\# \cF_\ell\geq \ell^2-\ell- 2 \left(\frac{\ell^2-\ell}{p-1}\right)\left(1+2.53 \frac{p(p+1)}{\sqrt{\ell}}\right).\]
     Consider the case when $p=5$. In this case, we get 
      \[\# \cF_\ell\geq \left(\frac{\ell^2-\ell}{2}\right)\left( \frac{\sqrt{\ell}-76}{\sqrt{\ell}}\right).\]
      This bound is nontrivial when $\sqrt{\ell}>76$, i.e., $\ell\geq 5779$. 
      \par On the other hand, consider the case when $p=3$ and assume that all primes that ramify in $L$ are $\equiv 1\pmod{4}$. In this case, $E_{\overline{A}, \overline{B}}\simeq E_{\overline{A}, -\overline{B}}$ over $\F_\ell$. Hence, we find that $\mathcal{A}_2=\mathcal{A}_3$. Therefore, $\mathcal{A}_1 \cup \mathcal{A}_2 \cup \mathcal{A}_3= \mathcal{A}_1 \cup \mathcal{A}_2$ and we have that 
     \[\begin{split} \# \cF_\ell \geq &
      \ell^2-\# \mathcal{A}_1-\#\mathcal{A}_2,\\
     \geq & \ell^2-\ell- \left(\frac{\ell^2-\ell}{2}\right)\left(1+ \frac{15.18}{\sqrt{\ell}}\right)\\
     = & \left(\frac{\ell^2-\ell}{2}\right)\left(\frac{\sqrt{\ell}-15.18}{\sqrt{\ell}}\right).
     \end{split}\]
     This bound is nontrivial when $\sqrt{\ell}>15.18$, i.e., $\ell\geq 233$.
\end{proof}

\begin{remark}\label{remark with code}
    To avoid confusion, we reiterate that in Theorem \ref{main thm}, it is assumed that for all $\ell\in Z$, there exists an elliptic curve $\mathbb{E}$ satisfying required properties. This in particular implies that $\delta_\ell>0$. However, the result above shows that this condition need only be checked for primes $\ell\geq 5779$ when $p=5$. If $p=3$ and additionally, one assumes that all primes $\ell\in Z$ are $1\pmod{4}$, then this condition needs to be checked only for $\ell\geq 233$. This has been checked by us using the following Sagemath code:
    \begin{verbatim}
P=Primes();
    for l in range(5,234):
        if l in P:
            c=0;
            for A in GF(l):
                for B in GF(l):
                    if (4*A^3+27*B^2)!=0:
                        E=EllipticCurve([A,B]);
                        E1=EllipticCurve([A,-B]);
                        if (E.abelian_group().order())%3 != 0:
                            if (E1.abelian_group().order())%3 != 0:
                                c=c+1;
            print(l,c)
\end{verbatim}
In fact, the code gives the exact values of $\#\overline{\cF}_\ell$. 
\end{remark}

For $z>0$, let $\cF^z(X)$ be the intersection $\cap_{\ell\leq z} \cF_\ell(X)$ where $\ell$ ranges over prime numbers that are $\leq z$. It is easy to see that since there are only finitely many congruence conditions, 
\begin{equation}\label{finitely many cong}\cF^z(X)\sim 4 \prod_{\ell\leq z} \delta_\ell X^{5/6}.\end{equation} 

\begin{proposition}\label{density of F is prod delta ell}
    With respect to notation above, the set $\cF$ has positive density given by $\mathfrak{d}(\cF)\sim \prod_{\ell} \delta_\ell$.
\end{proposition}
\begin{proof}
    We prove our result by showing that 
\begin{equation}\label{prop 4.15 eq 1}
    \limsup_{X\rightarrow \infty} \frac{\# \cF(X)}{4X^{5/6}}\leq \prod_{\ell} \delta_\ell.
    \end{equation}
    and 
    \begin{equation}\label{prop 4.15 eq 2}
    \liminf_{X\rightarrow \infty} \frac{\# \cF(X)}{4X^{5/6}}\geq \prod_{\ell} \delta_\ell.
    \end{equation}
We have that $\cF^z(X)\sim 4\prod_{\ell\leq z} \delta_\ell  X^{5/6}$. Since $\cF(X)$ is contained in $\cF^z(X)$, we find that 
     \[\limsup_{X\rightarrow \infty} \frac{\# \cF(X)}{4X^{5/6}}\leq \prod_{\ell\leq z} \delta_\ell,\]
     and letting $z\rightarrow \infty$, one has that
 \[\limsup_{X\rightarrow \infty} \frac{\# \cF(X)}{4X^{5/6}}\leq \prod_{\ell} \delta_\ell.\]
 Since \[\cF^z(X)\subseteq \cF(X)\cup \left(\bigcup_{\ell>z } \cF_\ell'(X)\right),\] where we recall that $\cF_\ell'$ is the complement of $\cF_\ell$. Lemma \ref{C_1 tech lemma} implies that there is an absolute constant $C_1>0$ for which
\[\begin{split}\liminf_{X\rightarrow \infty} \frac{\# \cF(X)}{4X^{5/6}}\geq & \lim_{X\rightarrow \infty} \frac{\# \cF^z(X)}{4X^{5/6}}-C_1 \sum_{\ell>z}\frac{1}{\ell^2},\\ = &\prod_{\ell \leq Z} \delta_\ell - C_1 \sum_{\ell>z}\frac{1}{\ell^2}. \end{split}\]
Letting $z\rightarrow \infty$, we find that 
\begin{equation}\label{liminf equation}\liminf_{X\rightarrow \infty} \frac{\# \cF(X)}{4 X^{5/6}}\geq \prod_{\ell} \delta_\ell.\end{equation}
This completes the proof of the result.
\end{proof}

 \subsection{Proof of the main result}
In this section, we provide the proof of our main result, Theorem \ref{main thm}.
\begin{proposition}\label{lower density of E0}
Setting \[\eta_p:=\begin{cases} \frac{1}{4} & \text{ if }p=3;\\
 \frac{3}{8} & \text{ if }p=5;\end{cases}\] we find that
    \[\liminf_{X\rightarrow \infty} \frac{\# \cE^0(X)}{\#\cE(X)}\geq \eta_p.\]
\end{proposition}

\begin{proof}
    The set $\cE$ is a large family and by Lemma \ref{lemma for root numbers}, exactly half the curves in $\cE$ have root number $1$. The result follows from Theorem \ref{thm of bhargava shankar}. 
\end{proof}

\begin{proof}[Proof of Theorem \ref{main thm}]
    \par Let $E$ be an elliptic curve in the set $\cE^0$. By Proposition \ref{ass 2.1 is satisfied prop}, the Assumption \ref{ass on good reduction at v in S} is satisfied for $E$ (and $K=\Q$). It follows from Proposition \ref{p selmer over L = 0} that $\op{Sel}_p(E/L)=0$. In particular, the rank of $E(L)$ is $0$. Let $S_3$ be the intersection $\cE^0\cap S_1$. Given $E\in S_3$, we find that $\op{rank}E(L)=0$ and $E(L)_{\op{tors}}=0$, hence, $E(L)=0$.
    
    \par Theorem \ref{S_1 has density 1} asserts that $S_1$ has density $1$.  We note that for $\ell\in Z$, the cardinality of $\overline{\cF}_\ell$ is $\mathfrak{A}_\ell$ by definition. Thus, we find that 
    \[\underline{\mathfrak{d}}(S_3)=\underline{\mathfrak{d}}(\cE^0).\]
    Then it follows from Proposition \ref{lower density of E0} that
    \[\underline{\mathfrak{d}}(\cE^0)\geq \eta_p \underline{\mathfrak{d}}(\cE).\]
    By Lemma \ref{F is subset of E}, $\cF$ is contained in $\cE$, and thus,
    \[ \eta_p \underline{\mathfrak{d}}(\cE)\geq \eta_p \underline{\mathfrak{d}}(\cF).\]
    Finally, note that according to  Proposition \ref{density of F is prod delta ell}, 
\[\mathfrak{d}(\cF)=\prod_\ell \delta_\ell.\]
Condition (2) implies that $\delta_\ell>0$ for all primes $\ell\in Z$. On the other hand, that $\delta_\ell>0$ for all primes $\ell\notin Z$ follows from  Lemma \ref{cl lemma}.
In conclusion, $\underline{\mathfrak{d}}(S_3) \geq \eta_p \prod_\ell \delta_\ell$, thus, the density of elliptic curves $E_{/\Q}$ that are diophantine stable in $L$ have positive lower density $\geq \eta_p \prod_\ell \delta_\ell$.
\end{proof}

\subsection{An example}
We illustrate Theorem \ref{main thm} via an example. We fix $p=3$, and $L$ to be the cubic subfield of $\Q(\mu_{31})$. Note that $2$ splits in $L$. The only prime that ramifies in $L$ is $31$, and thus $Z=\{31\}$ and $T=\{2,3,31\}$. Recall that $\mathfrak{A}_{31}$ is the number of $(\overline{A}, \overline{B})\in (\Z/31\Z)^2$ such that
\begin{itemize}
\item $4 \overline{A}^3+27 \overline{B}^2\neq 0$, 
\item $E_{\overline{A}, \overline{B}}(\F_{31})[3]=0$ and $E_{\overline{A}, -\overline{B}}(\F_{31})[3]=0$.
\end{itemize}
Computations on Sagemath below
\begin{verbatim}
flag = 0
for A in GF(31):
     for B in GF(31):
            if (4*A^3+27*B^2)%31 !=0:
                E=EllipticCurve([A,B])
                if (E.abelian_group().order())%3 != 0:
                    print("A=", A, " B=", B, " #E(A,B)=", E.order())
                    flag = flag +1
print(flag)
\end{verbatim}
show that $\mathfrak{A}_{31}=585$.
Thus, we find that the lower density of elliptic curves $E_{/\Q}$ that are diophantine stable in $L$ is 
\[\geq \eta_3\prod_\ell \delta_\ell=\frac{1}{4} \times \frac{1}{2^{21}}\times \left(1-\frac{1}{3}\right)\times \frac{585}{961}\times \prod_{\ell\neq 2,3,31} \left(1-\frac{2}{\ell^2}+\frac{1}{\ell^3}\right).\]

\bibliographystyle{alpha}
\bibliography{references}
\end{document}